\documentclass[11pt]{amsart}

\usepackage{amsmath,amssymb,amsfonts,amssymb,amsthm}
\usepackage{dsfont}
\usepackage{verbatim}
\usepackage[usenames]{color}
\usepackage{hyperref, enumitem}
\hypersetup{hidelinks}
\usepackage{url}
\usepackage{tikz,tikz-qtree,ifthen,cancel}
\usepackage{array,tikz-qtree,ifthen,cancel}
\usepackage{graphicx}
\usepackage{adjustbox}
\usepackage{amsthm,graphicx,tikz,appendix,tikz-qtree,ifthen,cancel, cleveref}
\usepackage{soul}
\usetikzlibrary{calc,shapes,patterns,positioning}

\usepackage{yfonts}

\usepackage{blindtext}

\usepackage{setspace}

\DeclareFontFamily{U}{mathb}{\hyphenchar\font45}
\DeclareFontShape{U}{mathb}{m}{n}{
      <5> <6> <7> <8> <9> <10> gen * mathb
      <10.95> mathb10 <12> <14.4> <17.28> <20.74> <24.88> mathb12
}{}
\DeclareSymbolFont{mathb}{U}{mathb}{m}{n}
\DeclareMathSymbol{\llcurly}{3}{mathb}{"CE}
\DeclareMathSymbol{\ggcurly}{3}{mathb}{"CF}
\DeclareFontFamily{U}{matha}{\hyphenchar\font45}
\DeclareFontShape{U}{matha}{m}{n}{
      <5> <6> <7> <8> <9> <10> gen * matha
      <10.95> matha10 <12> <14.4> <17.28> <20.74> <24.88> matha12
      }{}
\DeclareSymbolFont{matha}{U}{matha}{m}{n}
\DeclareMathSymbol{\curlywedge} {2}{matha}{"4E}
\DeclareMathSymbol{\curlyvee} {2}{matha}{"4F}
\DeclareMathOperator{\FIN}{FIN}

\newtheorem{thm}{Theorem}[section]
\newtheorem{prop}[thm]{Proposition}

\newtheorem{lem}[thm]{Lemma}
\newtheorem{cor}[thm]{Corollary}

\newtheorem*{thmMain}{Theorem \ref{thm.3.9}}

\newtheorem*{thmmainTRS}{Theorem \ref{maintheorem}}

\theoremstyle{remark}
\newtheorem{rem}[thm]{Remark}

\theoremstyle{definition}
\newtheorem{defn}[thm]{Definition}

\newtheorem{example}[thm]{Example}
\newtheorem{question}[thm]{Question}

\theoremstyle{remark}

\newcommand{\fin}{\mathrm{fin}}

\newcommand{\al}{\alpha}
\newcommand{\om}{\omega}

\newcommand{\sse}{\subseteq}

\DeclareMathOperator{\dom}{dom}

\DeclareMathOperator{\depth}{depth}

\newcommand{\bP}{\mathbb{P}}

\newcommand{\noprint}[1]{\relax}

\title{Tukey-Idempotency and Strong p-points}

\author{Tom Benhamou, Natasha Dobrinen, and Tan \"{O}zalp}
\thanks{The first author is supported by the NSF under Grant
No. DMS-2346680}
\thanks{The second author is supported by National Science Foundation Grant DMS-2300896}

\address[Benhamou]{Department of Mathematics, Rutgers University, New Brunswick, NJ, USA}
\email{tom.benhamou@rutgers.edu}
\address[Dobrinen]{Department of Mathematics, University of Notre Dame, Notre Dame, IN 46556, USA}
\email{ndobrine@nd.edu}
\address[\"{O}zalp]{Department of Mathematics, University of Notre Dame, Notre Dame, IN 46556, USA}
\email{aozalp@nd.edu}
\date{}

\begin{document}
\begin{abstract}
    We characterize strong $p$-point ultrafilters by showing that they are exactly those $p$-points that are not Tukey above $(\omega^\omega,\leq)$; or equivalently, those $p$-points that are not Tukey-idempotent. Moreover, we show that there are no Canjar ultrafilters on measurable cardinals. We make use of tools which were motivated by topological Ramsey spaces, developed in \cite{Benhamou/Dobrinen24}, and furthermore, show that ultrafilters arising from most of the known topological Ramsey spaces are Tukey-idempotent. Our results answer questions of Hru\v{s}\'ak and Verner \cite[Question 5.7]{Hrusak/Verner11}, Brook-Taylor \cite[Question 3.6]{{QuestionGeneralized}}, and partially Benhamou and Dobrinen \cite[Question 5.6]{Benhamou/Dobrinen24}.
\end{abstract}
\maketitle
\section{Introduction}
The study of combinatorial classes of ultrafilters on $\omega$ has been a major topic of interest, influencing many areas of mathematics. In this paper, we examine one of those well-studied classes, namely, the class of \textit{strong p-points} introduced by Laflamme in \cite{Laflamme89} and  Canjar in \cite{Canjar} (see Definition \ref{def: canjar}). This class consists of those ultrafilters that do not code a \textit{generic dominating real}, an extremely fast-growing sequence of natural numbers. The simplest codes are described via a fixed function $f:\omega\to\omega$,  by enumerating the set $f'' A$ for all $A$ in the ultrafilter. Both Canjar and Laflamme observed that strong p-points do not admit such enumerations which grow faster than any given function from $\omega$ to $\omega$.  They conjectured that in order to be a strong $p$-point, it suffices to avoid these types of enumerations. A counterexample to their conjecture was first constructed by Blass, Hru\v{s}\'ak and Verner \cite{BlassHrushaVerner}. In this paper, we present a characterization of strong $p$-points which captures the underlying intuition behind the Canjar-Laflamme conjecture. Our characterization allows more complex ways of coding a dominating real than plain enumerations of sets. The full scope of this complexity is naturally described using the \textit{Tukey order} by saying that $(\omega^\omega,\leq)$ is \textit{Tukey-reducible} to the ultrafilter, where $\leq$ refers to the \textit{everywhere domination} order of functions. Indeed, we show that Tukey reductions describe all possible codes of a generic dominating real by an ultrafilter.
Our main theorem is:
\begin{thmMain}
    Let $U$ be an ultrafilter on $\omega$, then:
    $$U\text{ is a strong }p\text{-point}\Longleftrightarrow (U,\supseteq)\not\geq_T(\omega^\omega,\leq).$$
\end{thmMain}

To state our result more precisely, we first recall the Tukey order \cite{Tukey40}. This concept originated in the study of Moore–Smith convergence of nets in topology. It is defined in greater generality: for two directed posets $(P,\leq_P)$ and $(Q,\leq_Q)$, we write $(P,\leq_P)\leq_T (Q,\leq_Q)$ if there is map $f:Q\rightarrow P$, which is cofinal, namely, $f''B$ is cofinal in $P$ whenever $B$ is cofinal in $Q$. Schmidt \cite{Schmidt55} observed that this is equivalent to having a map $f:P\rightarrow Q$, which is unbounded, namely, $f''\mathcal{A}$ is unbounded in $Q$ whenever $\mathcal{A}$ is unbounded in $P$.
We say that $P$ and $Q$ are {\em Tukey equivalent},
and write
$P\equiv_T Q$,
if $P\leq_T Q$ and $Q\leq_T P$;  the equivalence class $[P]_T$ is called the Tukey type or cofinal type of $P$.

The investigation of the Tukey type of posets of the form $(U,\supseteq)$, where $U$ is an ultrafilter on a countable set, started with the independent work of Isbell \cite{Isbell65} and Juh\'asz \cite{Juhasz66}, who constructed (in ZFC) ultrafilters of maximal Tukey type. More recently, this study has been an important line of research, experiencing a major development from a variety of authors. Examples include \cite{TomCommute, Benhamou/Dobrinen24, Blass/Dobrinen/Raghavan15, DobrinenJSL15,  Dobrinen/Todorcevic11, Milovich08, Ozalp24,  Raghavan/Shelah17,CancinoZaplatal}; and lately on uncountable cardinals as well \cite{BENHAMOU_DOBRINEN_2024,Benhamou_Goldberg_2025,benhamou2025ultrafilterssuccessorcardinalstukey}.

  In this paper, we will only use the characterization of strong $p$-point ultrafilters given by Canjar, namely, the fact that their corresponding Mathias forcing does not add dominating reals. The precise conjecture of Canjar and Laflamme was that being a $p$-point without rapid $RK$-predecessors is equivalent to being a strong $p$-point. Recall that an ultrafilter $U$ is \textit{rapid} if and only if the enumeration function $X\mapsto \pi_X$ (where $\pi_X:\omega\to X$ is the inverse of the transitive collapse of $X$ to $\text{otp}(X)$) is a cofinal map from $(U,\supseteq)$ to $(\omega^\omega,\leq)$, witnessing that $(U,\supseteq)\geq_T(\omega^\omega,\leq)$. Benhamou \cite{TomCommute} showed that
there are consistently $p$-point ultrafilters  which are not rapid but are Tukey above $(\omega^\omega,\leq)$ (e.g., $\alpha$-almost rapid ultrafilters).
This is one way of seeing that the property of being Tukey above $(\omega^\omega,\leq)$ generalizes the property of being $RK$-above a rapid ultrafilter, and in turn, how our main result relates to the Canjar-Laflamme conjecture.  

The property of being a $p$-point which is Tukey above $(\omega^\omega,\leq)$ turned out to be equivalent to \textit{Tukey-idempotency}. We say that an ultrafilter $U$ is \textit{Tukey-idempotent} if $U\cdot U\equiv_T U$, where $U\cdot U$ denotes the Fubini product of $U$ with itself (see, e.g., \cite{BlassThesis}). In \cite{Dobrinen/Todorcevic11}, Dobrinen and Todorcevic proved  that inside the class of $p$-points (see Definition \ref{def: i-p.i.p}), Tukey-idempotent ultrafilters are exactly those $U$ such that $(U,\supseteq)\geq_T (\omega^\omega,\leq)$.
 It is also known that non-$p$-points are always Tukey above  $(\omega^\omega,\leq)$ (see Theorem $4.2$ of \cite{TomCommute}). An unpublished  result of Dobrinen showed that any strong p-point constructed via a certain forcing of Laflamme in \cite{Laflamme89}
is not Tukey idempotent, 
further motivating us to find the following: Tukey-idempotency of a $p$-point is itself yet another characterization of not being a strong $p$-point (see Theorem \ref{thm:main characterization}).

The property of being a $p$-point which is Tukey above $(\omega^\omega,\leq)$ was  generalized in \cite{Benhamou/Dobrinen24}  via the notion of $I$-p.i.p., where $I$ is an ideal on  some countable base set $S$:
\begin{defn}[{\cite{Benhamou/Dobrinen24}}]\label{def: i-p.i.p}
    We say that an ultrafilter $U$ on a countable base set $S$  has the 
    {\em $I$-pseudo-intersection-property}, or just  the
\textit{$I$-p.i.p.}, if for every $\langle X_n\mid n<\omega\rangle\subseteq U$, there is $X\in U$ such that $X\setminus X_n\in I$ for each $n<\omega$. 
\end{defn}
The $I$-p.i.p.\ was motivated by work on ultrafilters forced by topological Ramsey spaces.  This will be made clear in \S \ref{Section: TRS}.
Note that an ultrafilter $U$ on $\om$  is a \textit{$p$-point} exactly when it has the $\mathrm{fin}$-p.i.p., where $\mathrm{fin}$ is the ideal of finite subsets of $\omega$. 
\begin{thm}[{\cite{Benhamou/Dobrinen24}}]\label{Tukeyidempotency}
   Assume that $U$ is an ultrafilter on a countable base set $S$, and suppose $I\subseteq U^*$ is an ideal on the same set $S$. If $U$ has the $I$-p.i.p.\ and $U\geq_T I^\omega$\footnote{Here, $I^\omega:=\prod_{n<\omega}I$ is ordered via pointwise inclusion.}, then $U$ is Tukey-idempotent.
\end{thm}
The pool of ultrafilters for which the theorem above can be applied to contains a variety of examples that arise in the literature as well as abstractly via the setup of \textit{topological Ramsey spaces} (TRS) (see \cite{TodorcevicBK10}). There are  standard methods for 
 deriving an ultrafilter on a countable base set from a TRS, first pointed out by Todorcevic, and those ultrafilters satisfy various partition properties (see 
\cite{Dobrinen/NavarroFlores, Mijares07, Trujillo16,Zheng17}).
Benhamou and Dobrinen showed in \cite{Benhamou/Dobrinen24} that ultrafilters arising from high and infinite-dimensional Ellentuck spaces \cite{DobrinenJSL15,DobrinenJML16} as well as ultrafilters arising from the TRS of infinite block sequences are Tukey-idempotent.
In \S \ref{Section: TRS}, we provide general conditions $(*)$ (Definition \ref{defn.star})
under which  Tukey-idempotency is indeed a virtue of ultrafilters arising from  TRS's. 
\begin{thmmainTRS}
Let $\mathcal{R}$ be a topological Ramsey space, 
and suppose there is a proper  ideal $I\subseteq \mathcal{AR}_1$  satisfying  property $(*)$.
Then the generic
filter of 
first approximations
 $U_1$ forced by  $(\mathcal{R},\le)$ is a  Tukey-idempotent ultrafilter.
\end{thmmainTRS}
The assumption $(*)$ holds for all  known 
topological Ramsey spaces that have $\sigma$-closed separative quotients which force ultrafilters on $\mathcal{AR}_1$.
Examples include Ramsey ultrafilters (forced by the Ellentuck space),
all ultrafilters forced by $\mathcal{P}(\om^{\al})/\fin^{\otimes \al}$, $1\le \al<\om_1$,
(equivalently, forced by high and infinite-dimensional Ellentuck spaces \cite{DobrinenJSL15,DobrinenJML16}),
$k$-arrow ultrafilters of Baumgartner and Taylor \cite{Baumgartner/Taylor78}  and more generally, all ultrafilters forced by TRS's in \cite{Dobrinen/Mijares/Trujillo14}, and the higher order stable ordered union ultrafilters (forced by $\FIN^{[\infty]}_k$, $k\ge 1$).

\vskip 0.2 cm
 In \S \ref{Section: Applications}, we provide applications of our results and prove several immediate corollaries regarding: $p$-points with small generating sets, simple $P_\lambda$-points, and structural properties of the class of Canjar ultrafilters. We then present our main application to the theory of adding ultrafilters by forcings of the type $\mathcal{P}(\omega)/I$, where $I$ is a definable ideal on $\omega$. Hru\v{s}\'ak and Verner in \cite{Hrusak/Verner11}, used this to obtain another counterexample to the Canjar-Laflamme conjecture. They forced a generic ultrafilter for $P(\omega)/I$, where $I$ is any tall locally $F_\sigma$-ideal. They asked the following question:

\begin{question}[Question $5.7$ of \cite{Hrusak/Verner11}]\label{Q: HrusakVerner}
Is there a Borel ideal $I$ on $\omega$ such that $\mathcal{P}(\omega)/I$ adds a Canjar ultrafilter? 
\end{question}
We answer\footnote{ Guzm\'{a}n and Hru\v{s}\'{a}k informed us that they have answered this question independently using different methods in unpublished work.} this question negatively, even for analytic ideals $I$. This is a direct corollary (see Corollary \ref{cor: Answer question}) of the following theorem:
\begin{thmMain}
     Let $I$ be an analytic ideal on $\omega$ such that $P(\omega)/I$ does not add reals. Then $P(\omega)/I$ adds an ultrafilter $U$ which is Tukey above $(\omega^\omega,\leq)$.
\end{thmMain}
Finally, in \S\ref{Section:Measurable}, we show that there are no Canjar ultrafilters on measurable cardinals. This provides an answer to a question from \cite{QuestionGeneralized} attributed to Brook-Taylor. In the last section \S\ref{Section: Questions}, we present several related open problems.

\subsection*{Notation}\label{Section:notations} 
For any set $X$ and a cardinal $\lambda$, $[X]^{<\lambda}$ denotes the set of all subsets of $X$ of cardinality less than $\lambda$. Let $\fin=[\omega]^{<\omega}$, and $\FIN=\fin\setminus\{\emptyset\}$. Let $\omega^{<\omega}$ the collection of finite partial functions $f:\omega\to\omega$.
For a collection of sets $(P_i)_{i\in I}$ we let $\prod_{i\in I}P_i=\{f:I\rightarrow \bigcup_{i\in I}P_i\mid (\forall i)\, f(i)\in P_i\}$. Given a set $X\subseteq \omega$ such that $|X|=\alpha\leq\omega$, we denote by $\langle X(\beta)\mid \beta<\alpha\rangle$  the increasing enumeration of $X$.
Given a function $f:A\rightarrow B$, for $X\subseteq A$ we let $f''X=\{f(x)\mid x\in X\}$ and for $Y\subseteq B$ we let $f^{-1}Y=\{x\in X\mid f(x)\in Y\}$. Given an ultrafilter $U$ on $X$ and $f:X\to Y$ we denote by $f_*(X)=\{A\subseteq Y\mid f^{-1}A\in U\}$. We say that $W\leq_{RK}U$ if there is $f$ such that $f_*(U)=W$  and $U\equiv_{RK} W$ if $f$ is one-to-one on some set in $U$.

\section{The main result}
Recall that, given an ultrafilter $U$ on $\omega$, \emph{Mathias forcing restricted to $U$}, denoted $\mathbb{M}_U$, consists of conditions of the form $\langle a,A\rangle\in [\omega]^{<\omega}\times U$ with $\min(A)>\max(a)$. The order on $\mathbb{M}_U$ is defined by $\langle a,A\rangle \leq \langle b,B\rangle$ if and only if $b\sqsubseteq a$, $a\setminus b\subseteq B$, and $A\subseteq B$.
\begin{defn}
Let $\mathbb{P}$ be a notion of forcing. $\mathbb{P}$ \emph{does not add dominating reals} if
$$
\mathds{1}_{\mathbb{P}}\Vdash (\forall \dot{f}\in\omega^\omega)(\exists g\in\omega^\omega\cap {\mathbf V})(\exists^{\infty}n)(\dot{f}(n)\leq g(n)),
$$
where ${\mathbf V}$ denotes the ground model.
\end{defn}
\begin{defn}\label{def: canjar}
    An ultrafilter $U$ is called \emph{Canjar} if  $\mathbb{M}_U$ does not add a dominating real.
\end{defn}
Canjar \cite{Canjar} proved that such ultrafilters exist under $\mathfrak{d}=\mathfrak{c}$.
\begin{prop}[\cite{Canjar}]\label{Prop: Canjar}\ {}
\begin{enumerate}
    \item If $U$ is Canjar and $U\geq_{RK}V$, then $V$ is Canjar.
    \item If $U$ is Canjar, then $U$ is a $p$-point with no rapid $RK$-predecessors. 
\end{enumerate}
\end{prop}
Hru\v{s}\'ak and Minami \cite{HrusakMinami2014} found a combinatorial characterization of Canjar ultrafilters which we describe next.
\begin{defn}
    Given an ultrafilter $U$ on $\omega$, we define $U^{<\omega}$ to be the filter on $\FIN$ generated by the sets $\{[A]^{<\omega}\mid A\in U\}$.
\end{defn}
We say that a filter $F$ on a set $X$ is a $P^+$-filter if for any $\subseteq$-decreasing sequence $\langle X_n\mid n<\omega\rangle\subseteq F^+$ there is $X_{\infty}\in F^+$ such that $\forall n \ |X_{\infty}\setminus X_n|<\aleph_0$. We call $X_{\infty}$ a \textit{pseudo-intersection} of the sequence $\langle X_n\mid n<\omega\rangle\subseteq F^+$.

In \cite{HrusakMinami2014}, it was shown that an ultrafilter $U$ on $\omega$ is Canjar if and only if $U^{<\omega}$ is a $P^+$-filter.

Now we are ready to  state and prove our main characterization:
\begin{thm}\label{thm:main characterization}
The following are equivalent for an ultrafilter $U$ on $\omega$:
\begin{enumerate}[label=(\roman*)]
    
    \item $U\cdot U>_T U$ and $U$ is a $p$-point. \label{thm:main characterizatio1}
    \item $(U,\supseteq)\not\geq_T (\omega^\omega,\leq)$. \label{thm:main characterizatio2}
    \item $U$ is Canjar. \label{thm:main characterizatio3}
    \item $U$ is a strong $p$-point. \label{thm:main characterizatio4}
    \item $U^{<\omega}$ is a $P^+$-filter. \label{thm:main characterizatio5}
\end{enumerate}
\end{thm}
 The equivalence between \ref{thm:main characterizatio3} and \ref{thm:main characterizatio4} was established by Blass, Hru\v{s}\'ak and Verner \cite{BlassHrushaVerner}. Hru\v{s}\'ak and Minami proved the equivalence of \ref{thm:main characterizatio4} and \ref{thm:main characterizatio5} in \cite{HrusakMinami2014}. Dobrinen and Todorcevic proved the implication from \ref{thm:main characterizatio1} to \ref{thm:main characterizatio2}, and that \ref{thm:main characterizatio2} entails $U\cdot U>_T U$. For the implication $U\not\geq_T \omega^\omega\Rightarrow U$ is a $p$-point, we refer the reader to \cite{TomCohesive}. Hence, our contribution is the equivalence between \ref{thm:main characterizatio2} and \ref{thm:main characterizatio3}.
Let us start by proving that \ref{thm:main characterizatio2} implies \ref{thm:main characterizatio3}. The other direction is proven in \ref{finalstep}.
\begin{prop}\label{prop: the easy direction}
    If $U$ is not Canjar, then $(U,\supseteq)\geq_T (\omega^\omega, \leq)$.
\end{prop}

\begin{proof}
    For $Z\subseteq\FIN$ and $n_0\in\omega$, we denote $Z\setminus n_0=\{s \in Z : \max(s)<n_0\}$. Assume that $U$ is not Canjar. Since \ref{thm:main characterizatio3} and \ref{thm:main characterizatio5} are equivalent \cite{HrusakMinami2014}, there is some $\subseteq$-decreasing sequence $\langle X_n\mid n<\omega\rangle\subseteq(U^{<\omega})^+$  that witnesses $U^{<\omega}$ not being a $P^+$-filter. For $Z \in U$, we define $f_Z\in\omega^\omega$ by $$f_Z(n)=\min\{\max(s)\mid s\in[Z]^{<\omega}\cap X_n\}.$$ The map from $U$ to $\omega^\omega$ given by $Z \mapsto f_Z$ is a monotone map which we now prove is cofinal. For that purpose, take an arbitrary $g \in \omega^\omega$. Note that the set $\bigcup_{n\in\omega} (X_n\setminus g(n))$ is a pseudo intersection of the collection of the $X_n$'s; thus it cannot be in $(U^{<\omega})^+$. Consequently, there is some $Y_g \in U$ such that \begin{equation}(\bigcup_{n\in\omega} X_n\setminus g(n))\cap [Y_g]^{<\omega}=\emptyset.\end{equation}
    It follows that $g(n)\leq f_{Y_g}(n)$ for all $n\in\omega$, and we are done.
\end{proof}
For the other direction, we will use a stronger form of the Tukey order which will turn out to be the same in our context:
\begin{defn}\label{Tukey*}
We say that $U$ is \textit{canonically Tukey above $(\omega^\omega, \leq^*)$}\footnote{Here, $\leq^*$ denotes the \emph{eventual domination} order on $\omega^\omega$.} if there is $f:U\to \omega^\omega$ and $\hat{f}:[\omega]^{<\omega}\to \omega^{<\omega}$ such that:
\begin{enumerate}[label=(\alph*)]
    \item $f:(U, \supseteq)\to (\omega^\omega, \leq^*)$ is monotone and cofinal.\label{Tukey*0}
    \item Let $s,t\in [\omega]^{<\omega}$. If $s\subseteq t$, then $\dom(\hat{f}(s))\subseteq \dom(\hat{f}(t))$, and moreover for every $x\in \dom (\hat{f}(s))$, $\hat{f}(s)(x)\geq \hat{f}(t)(x)$. In this case, we write $\hat{f}(s)\geq \hat{f}(t)$.\label{Tukey*1}
    \item For every $\langle s,A\rangle \in \mathbb{M}_U$, and every $m\in\omega$, there is some $t\in [A]^{<\omega}$ such that $m\in\dom(\hat{f}(s\cup t))$.\label{Tukey*2}
    \item\label{Tukey*3} There is some $X\in U$ such that for every $Y\in U\restriction X$ and $k \in \omega$, we have \begin{equation} f(Y)(k)=\min\{\hat{f}(Y\cap n)(k)\mid n\in\omega, \ k\in\dom(\hat{f}(Y\cap n)\}.\end{equation}
\end{enumerate}
\end{defn}
\begin{rem}\label{remark}
By condition \ref{Tukey*1}, the value $f(Y\cap n)(k)$ in \ref{Tukey*3} stabilizes, namely,  for all $Y\in U\upharpoonright X$ and $k\in \omega$, $(\exists n=n(Y,k))(\forall m\geq n)$ $\hat{f}(Y\cap n)(k)=\hat{f}(Y\cap m)(k)=f(Y)(k)$.
\end{rem}
\begin{example}
Let us now give some examples that motivate this definition. In the following exmples, the $X\in U$ of condition \ref{Tukey*3} of \Cref{Tukey*} is going to be $X=\omega$.
\begin{enumerate}[label=(\roman*)]
    \item If $U$ is rapid,   for $Y\in U$ and  $s\in[\omega]^{<\omega}$, we let $f(Y)=\pi_Y$,  $\hat{f}(s)=\pi_s$ (Recall that $\pi_X$ denote the enumerating function). Note that if $Y\in U$ and $n<m<\omega$, then $Y\cap n\sqsubseteq Y\cap m$, and thus, in condition \ref{Tukey*3} is satisfied.
    \item More generally, if $\varphi:\omega\to\omega$ is a map and the $RK$-image $\varphi_*(U)$ is rapid, then for $Y\in U$ and $s\in [\omega]^{<\omega}$, we can set $f(Y)=\pi_{\varphi[Y]}$, and $\hat{f}(s)=\pi_{\varphi[s]}$. If $s\subseteq t$, then $\varphi[s]\subseteq \varphi[t]$, and  condition \ref{Tukey*1} of \Cref{Tukey*} is satisfied. Let us now verify condition \ref{Tukey*3}. Note that it is possible that $n<m$ are both in $Y\in U$, but also $\varphi(m)<\varphi(n)$. In this case, if $\varphi(n)=\pi_{\varphi[Y\cap n+1]}(k)$ then we would have $\pi_{\varphi[Y\cap m+1]}(k)<\pi_{\varphi[Y\cap n+1]}(k)$. However, by Remark \ref{remark},  $\hat{f}(Y\cap n)(k)$ stabilizes. Hence, the term in condition \ref{Tukey*3} ends up being $\pi_{\varphi[Y]}=f(Y)$. 
\end{enumerate}
\end{example}
\begin{rem}
    Note that if $U$ is canonically above $(\omega^\omega,\leq^*)$ then $U\geq_T(\omega^\omega,\leq^*)$, which implies that $U\geq_T (\omega^\omega,\leq^*)\times\omega\equiv_T (\omega^\omega,\leq)$.
\end{rem}
Let us now strengthen \Cref{prop: the easy direction} to our new definition.
\begin{lem}
    If $U$ is not Canjar, then $(U,\supseteq)$ is canonically Tukey above $(\omega^\omega,\leq^*)$.
\end{lem}
\begin{proof}
Assume that $U$ is not Canjar. Let $\dot{f}$ be a $\mathbb{M}_U$-name such that $\mathds{1}_{\mathbb{M}_U}\Vdash ``\dot{f}$ is a dominating real". As in $(1)\Rightarrow (2)$ in \cite[Theorem 3.8]{HrusakMinami2014}, a pigeonhole argument guarantees that for some $n^*\in\omega$ and $s\in [\omega]^{<\omega}$ there is a dominating family $\mathcal{F}\in V$, such that for every $g\in\mathcal{F}$,  \begin{equation}
    \langle s, F_g\rangle\Vdash\forall m\geq n^*, \dot{f}(m)>g(m).
\end{equation}
For some $F_g\in U$.
For $B\in U$, we define $f_B\in\omega^\omega$ as follows. For $n\in\omega$, \begin{equation}\label{lemma3.11preq}
f_B(n)=\min\{m\in\omega\mid(\exists t\in [B\setminus \max(s)+1]^{<\omega})(\exists C\in U) \ (\langle s\cup t,C\rangle\Vdash \dot{f}(n)=m)\}.\end{equation}
Note that in \eqref{lemma3.11preq}, the value forced by $\langle s\cup t,C\rangle$ to be $\dot{f}(n)$ does not depend on the choice of $C\in U$.

Let us first show that the function $f:U\to \omega^\omega$, defined by $f(B)=f_B$ for all $B\in U$, is monotone and cofinal. As monotonicity of $f$ is routine, let us focus on $f$ being cofinal. Take an arbitrary $h\in  \omega^\omega$. Find $g\in\mathcal{F}$ such that $h\leq^* g$. By our assumption on $\mathcal{F}$, there must be some $F_g\in U$ such that $\langle s,F_g\rangle\Vdash (\forall m\geq n^*) \ g(m)\leq \dot{f}(m)$. Hence, for every $m\geq n^*$, and every $t\in [F_g]^{<\omega}$, if there is $C\in U$ such that $\langle s\cup t,C\rangle \Vdash \dot{f}(m)=k$, then we also have $\langle s\cup t,C\cap F_g\rangle \Vdash \dot{f}(m)=k$. But since $\langle s\cup t,C\cap F_g\rangle \leq (s,F_g)$, we must have $k\geq g(m)$; thus $f_{F_g}(m)\geq g(m)$. Therefore, $h\leq^* g\leq^* f(F_g)$ as wanted.

 Let us moreover show that $U$ canonically Tukey above $(\omega^\omega,\leq^*)$ as witnessed by $f$ and $\hat{f}$ by defining an appropriate $\hat{f}: [\omega]^{<\omega}\to \omega^{<\omega}$. To this end, for $t\in [\omega]^{<\omega}$, we define \begin{equation}
z_t=\{m\in\omega\mid (\exists k\in\omega)(\exists t'\subseteq t)(\exists C\in U)\ (\langle s\cup t',C\rangle \Vdash \dot{f}(m)=k)\}.
\end{equation}
For $t\in[\omega]^{<\omega}$, we set $\dom(\hat{f}(t))=z_t$, and for all $m\in z_t$, define \begin{equation}\hat{f}(t)(m)=\min\{k\in\omega \mid (\exists t'\subseteq t)(\exists C) \ (\langle s\cup t',C\rangle\Vdash \dot{f}(m)=k)\}.\end{equation}

Let us check conditions \ref{Tukey*1}-\ref{Tukey*3} of \Cref{Tukey*}. For \ref{Tukey*1}, if $t_0,t_1\in [\omega]^{<\omega}$ and $t_0\subseteq t_1$ holds, then we have $z_{t_0}\subseteq z_{t_1}$. Moreover, for every $m\in z_{t_0}$, the minimum in the definition of $\hat{f}(t_1)(m)$ is taken over more elements than $\hat{f}(t_0)(m)$, so \ref{Tukey*1} is satisfied. To verify \ref{Tukey*2}, fix $m\in\omega$ and $\langle t,A\rangle\in\mathbb{M}_U$. The condition $\langle s\cup t,A\setminus \max(s\cup t)+1\rangle\in\mathbb{M}_U$ has an extension which decides the value $\dot{f}(m)$, in other words, there is $r\in [A\setminus \max(s\cup t)+1]^{<\omega}$ such that $m\in z_{t\cup r}=\dom(\hat{f}(t\cup r))$. Finally, to verify \ref{Tukey*3}, set $X=\omega$ and take $Y\in U$. Let $k\in \omega$ be arbitrary. By Remark \ref{remark}, find some $n_0\in\omega$ such that $\hat{f}(Y\cap m)(k)$ is minimal for every $m\geq n_0$. By definition of $f_Y(k)$, there is some $t\in [Y\setminus \max(s)+1]^{<\omega}$ and $C\in U$ such that $\langle s\cup t, C\rangle \Vdash \dot{f}(k)=l$ for some $l\in\omega$, and $f_Y(k)=l$. Therefore, $t\subseteq Y\cap m$ for some $m\geq n_0,\max(t)$, and so $f_Y(k)=\hat{f}(Y\cap m)(k)$.
\end{proof}

\begin{lem}\label{lem:notCanjarcharacterization}
For any ultrafilter $U$ on $\omega$, the following conditions are equivalent:
\begin{enumerate}[label=(\roman*)]
    \item $U$ is not Canjar. \label{lem:notCanjarcharacterization1}
    \item  $(U,\supseteq)$ is canonically Tukey above $(\omega^\omega, \leq^*)$. \label{lem:notCanjarcharacterization2}
\end{enumerate}
\end{lem}
\begin{proof}
We have already proved the direction \ref{lem:notCanjarcharacterization1} $\Rightarrow$ \ref{lem:notCanjarcharacterization2}. Assume that $U$ is canonically Tukey above $(\omega^\omega,\leq^*)$ as witnessed by some $f$ and $\hat{f}$. In particular, $f,\hat{f}$ satisfy \ref{Tukey*0}-\ref{Tukey*3} of \Cref{Tukey*}. To see that $U$ is not Canjar, let $G$ be a generic filter for $\mathbb{M}_U$, and let $X_G$ denote the Mathias generic real for $\mathbb{M}_U$. Enumerate $[\omega]^{<\omega}=\langle t_n\mid n\in\omega\rangle$. In $V[G]$, set $B_n=t_n\cup X_G\setminus \max(t_n)+1$, and define a sequence of functions $\langle h_n : n\in\omega\rangle\subseteq \omega^\omega$ as follows: For $n,k\in\omega$, define \begin{equation}
h_n(k)=\min\{\hat{f}(B_n\cap m)(k)\mid m\in\omega, \ k\in \dom(\hat{f}(B_n\cap m))\}. 
\end{equation}
To see that $h_n$ is total, we proceed with a density argument. Let $m\in\omega$ and fix a condition $\langle s,A\rangle \in\mathbb{M}_U$. As $\langle t_n\cup (s\setminus \max(t_n)+1), A\setminus\max(t_n\cup s)+1\rangle\in\mathbb{M}_U$, by \ref{Tukey*2} of \Cref{Tukey*}, 
$$\exists t\in [A\setminus (\max(t_n\cup s)+1)]^{<\omega}\text{ such that }m\in\dom(\hat{f}(t_n\cup (s\setminus \max(t_n)+1)\cup t)).$$ Extend $\langle s, A\rangle$ to $\langle s\cup t,A\rangle$. We have $$\langle s\cup t,A\rangle \Vdash B_n\cap \max(t)+1=t_n\cup (s\setminus \max(t_n)+1)\cup t,$$ and therefore $\langle s\cup t,A\rangle \Vdash m\in \dom(h_n)$. 
    
Knowing that $h_n\in\omega^\omega\cap V[G]$ for all $n\in\omega$, we can find some $h\in V[G]$ such that $h_n\leq^* h$, for all $n\in\omega$. We claim that $h\in V[G]$ is a dominating real over $\mathbf{V}$. To see this, let $g\in \omega^\omega\cap \mathbf{V}$ be arbitrary, and find $B\in U$ such that $f(B)\geq^* g$. Find $N\in \omega$ such that $X_G\setminus N\subseteq B$. There is some $n\in\omega$ such that $t_n=B\cap N$, and therefore $B_n=(B\cap N)\cup (X_G\setminus N)\subseteq B$. Since $B_n\cap k\subseteq B\cap k$ for every $m\in\omega$, \ref{Tukey*1} of \Cref{Tukey*} implies that $\hat{f}(B_n\cap m)\geq \hat{f}(B\cap m)$ for every $m\in\omega$. In particular, whenever $k\in \dom(\hat{f}(B_n\cap m))$ for some $m\in\omega$, we also have $k\in \dom(\hat{f}(B\cap m))$ and $\hat{f}(B_n\cap m)(k)\geq \hat{f}(B\cap m)(k)$. This means that $h_n(k)\geq  f(B)(k)$ for all $k\in\omega$. Hence $h\geq^* h_n\geq f(B)\geq^* g$.
\end{proof}
Let us now prove the remaining implication of \Cref{thm:main characterization} (Namely \ref{thm:main characterizatio2}$\Rightarrow$\ref{thm:main characterizatio1}):
\begin{prop}\label{finalstep}
For any ultrafilter $U$ on $\omega$, if $(U,\supseteq)\geq_T(\omega^\omega, \leq)$ then $U$ is not Canjar.
\end{prop}
\begin{proof}   
Suppose that $(U,\supseteq)\geq_T(\omega^\omega, \leq)$. If $U$ is not a $p$-point, then $U$ is necessarily not Canjar; thus we may assume that $U$ is a $p$-point. We will show that for a $p$-point $U$, $(U,\supseteq)\geq_T(\omega^\omega, \leq)$ implies that $U$ is canonically Tukey above $(\omega^\omega, \leq^*)$, which will be enough to finish by the equivalence from Lemma \ref{lem:notCanjarcharacterization}.

Since $(U,\supseteq)\geq_T(\omega^\omega, \leq)$, by Dobrinen and Todorcevic's result \cite[Thm. 35]{Dobrinen/Todorcevic11}, we have $U\equiv_T U\cdot U\equiv_T U^{\omega}$. Again, by Dobrinen and Todorcevic \cite[Thm. 20]{Dobrinen/Todorcevic11}, there is a monotone and cofinal\footnote{Here, $U^{\omega}:=\prod_{n\in\omega}U$ is ordered via pointwise reverse inclusion.} $\varphi:U\to U^\omega$ continuous on $\omega\times \omega$ (see \Cref{lem:continuouscofinalmaps}), $\psi:[\omega]^{<\omega}\to [\omega\times\omega]^{<\omega}$ monotone, and some $X\in U$, such that for every $Y\in U\restriction X$, we have $\varphi(Y)=\bigcup_{n\in\omega}\psi(Y\cap n)$. Next, define the functions $f$ and $\hat{f}$. For $Y\in U$, define $$f(Y):\omega\to\omega\text{ by }f(Y)(n)=\min(\varphi(Y)_n),$$ and define $\hat{f}:[\omega]^{<\omega}\to \omega^{<\omega}$ by setting $$\dom(\hat{f}(s))=\{n\in\omega \mid \psi(s)_n\neq\emptyset\}\text{ and defining }\hat{f}(s)(n)=\min(\psi(s)_n).$$
To see that $f$, let $g\in\omega^\omega$ be arbitrary, and note that the sequence of final intervals $\langle(g(n),\omega)\mid n\in\omega\rangle$ is in $U^\omega$. Therefore there is some $Y\in U$ such that $\varphi(Y)_n\subseteq(g(n),\omega)$ for all $n$, hence $\min(\varphi(X)_n)\geq g(n)$.

Next we check \ref{Tukey*1}-\ref{Tukey*3} of \Cref{Tukey*}. For \ref{Tukey*1}, note that by monotonicity, if $s\subseteq t$, then $\psi(s)\subseteq \psi(t)$ holds. Therefore, $\dom(\hat{f}(s))\subseteq \dom(\hat{f}(t))$, and for every $n\in \dom(\hat{f}(s))$, we have $\hat{f}(s)(n)\geq \hat{f}(t)(n)$. For \ref{Tukey*2}, let $\langle s,A\rangle\in\mathbb{M}_U$ be any condition, and let $m\in\omega$ be arbitrary. Consider $\varphi(s\cup A)\in U^\omega$. We can find some large enough $k\geq m$ so that $\varphi(s\cup A)_m\cap k\neq\emptyset$. Hence, there is some $r\in\omega$ such that for $t=A\cap r$, we have $\psi(s\cup t)\cap (k\times k)=\varphi(s\cup A)\cap (k\times k)$. In particular, $\psi(s\cup t)_m\cap k\neq \emptyset$, and so $m\in \dom(\hat{f}(s\cup t))$. Condition \ref{Tukey*3} is  similar and from the definitions of $f$, $\hat{f}$ and the continuity.
\end{proof}
We include the following lemma; its proof  is essentially the proof of \cite[Thm. 20]{Dobrinen/Todorcevic11}, due to Dobrinen and Todorcevic.
Identify each sequence $\langle A_n\mid n\in\omega\rangle\in \prod_{n\in\omega}P(\omega)$ with a subset of $\omega\times \omega$ in the canonical way, and recall that $\prod_{n\in\omega}P(\omega)$ is ordered via pointwise reverse inclusion.
\begin{lem}\label{lem:continuouscofinalmaps}
Let $U$ be a $p$-point on $\omega$. Suppose that $\varphi:U\to \prod_{n\in\omega}P(\omega)$ is a monotone map. Then, there is some $X_*\in U$, and $\psi:[\omega]^{<\omega}\to [\omega\times \omega]^{<\omega}$ such that:
\begin{enumerate}[label=(\roman*)]
    \item $\varphi(Y)=\bigcup_{n\in\omega}\psi(Y\cap n)$, for any $Y\in U\restriction X_*$.
    \item $\psi$ is monotone.
\end{enumerate}
\end{lem}
\section{Applications}\label{Section: Applications}
Using the characterization of Theorem~\ref{maintheorem}, we are able to improve some known results (see Proposition~\ref{Prop: Canjar}). The first regards the structure of the class of Canjar ultrafilters. 
\begin{cor}
    If $U$ is Canjar and $U\geq_T V$, then $V$ is also Canjar.
\end{cor}
The second regards the $RK$-predecessors of a Canjar ultrafilter. The first author constructed extension of the class of rapid ultrafilters called \textit{$\alpha$-almost rapid} (see \cite[Definition 4.11]{TomCommute}), all Tukey above $(\omega^\omega,\leq)$ (see \cite[Proposition 4.13]{TomCommute}) and proved that consistently there are $p$-point which are $\alpha$-almost rapid but not rapid (see \cite[Theorem 4.15]{TomCommute}). 

\begin{cor}
    If $U$ is Canjar and $V$ is $\alpha$-almost rapid for some $\alpha<\omega_1$, then $U\not\geq_T V$.
\end{cor}
The theorem also provides simple sufficient condition for Canjarness:
\begin{cor}\label{corollarywithtwoparts}
Let $U$ be an ultrafilter on $\omega$. 
    \begin{enumerate}[label=(\roman*)]
        \item If $\chi(U)<\mathfrak{d}$, then $U$ is Canjar.\label{corollarywithtwoparts1}
        \item If $U$ is a $P_{\mathfrak{b}^+}$-point, then $U$ is Canjar. \label{corollarywithtwoparts2}
    \end{enumerate}
\end{cor}
\begin{proof} For both parts we prove the contrapositive. Suppose that $U$ is not Canjar. By \Cref{thm:main characterization}, there is a monotone cofinal map $f:(U,\supseteq)\to(\omega^\omega,\leq)$, and so $\chi(U)\geq \mathfrak{d}$ directly follows and so does Part \ref{corollarywithtwoparts1}. For \ref{corollarywithtwoparts2}, suppose towards a contradiction that $U$ is  a $P_{\mathfrak{b}^+}$-point.  Since  $f''U\subseteq \omega^\omega$ is a dominating family, we can find an unbounded family $\langle b_i\mid i<\mathfrak{b}\rangle\subseteq f''U$ which can moreover assumed to be $\leq^*$-increasing.  Find $\mathcal{X}\subseteq U$ of size $\mathfrak{b}$ such that $f''\mathcal{X}=\mathcal{B}$. By the assumption, there is some $A\in U$ which is almost included in every $X\in\mathcal{X}$, and since $\mathfrak{b}$ is regular uncountable there is $n<\omega$ and $\mathcal{X}'\in [\mathcal{X}]^{\mathfrak{b}}$ such that $A\setminus n\subseteq \bigcap \mathcal{X}'$. It follows that $\varphi=f(A\setminus n)$ will dominate $b_i$ for unboundedly many $i<\mathfrak{b}$. Since the sequence $\langle b_i\mid i<\mathfrak{b}\rangle$ is increasing it follows that $\varphi$ eventually dominates every $b_i$, contradicting the unboundedness assumption regarding the sequence. 
\end{proof}
Recall that for an uncountable cardinal $\kappa$, an ultrafilter $U$ on $\omega$ is called a \textit{$P_{\kappa}$-point} if every $\subseteq^*$-descending sequence of members of $U$ of length $<\kappa$ has a pseudo-intersection in $U$. A $P_{\kappa}$-point $U$ is called \textit{simple} if it is moreover true that $\chi(U)=\kappa$.
In \cite{BM-simplepoints}, Br\"auninger and Mildenberger solved the long-standing open problem of showing the consistency of the statement: ``there is a simple $P_{\aleph_1}$-point $U$ and a simple $P_{\aleph_2}$-point $V$''. In their model, $U$ and $V$ both turned out to be Canjar. The following corollary shows that this is inevitable:
\begin{cor}
If $\chi(U)=\mu<\lambda$ and $V$ is a $P_{\lambda}$-point, where $\mu$ and $\lambda$ are uncountable cardinals, then $U$ and $V$ are both Canjar. In particular, if $U$ is a simple $P_{\mu}$-point and $V$ is a simple $P_{\lambda}$-point, where $\mu\neq\lambda$ are uncountable cardinals, then $U$ and $V$ are both Canjar.
\end{cor}
\begin{proof}
It is well-known that $\omega_1\leq\mathfrak{b}\leq\mathfrak{u}$. Moreover, if there is a $P_{\kappa}$-point, then $\kappa\leq \mathfrak{s}\leq \mathfrak{d}$ (see \cite{Nyikos}). Therefore, if $\chi(U)=\mu$ and $V$ is a $P_{\lambda}$-point for some infinite cardinals $\mu<\lambda$, then $\omega_1\leq \mathfrak{b}\leq\mathfrak{u}\leq \mu<\lambda\leq\mathfrak{d}$. Consequently, by \Cref{corollarywithtwoparts} part \ref{corollarywithtwoparts1}, $U$ has to be Canjar because $\chi(U)<\mathfrak{d}$. Also, by \Cref{corollarywithtwoparts} part \ref{corollarywithtwoparts2}, $V$ has to be Canjar because $V$ is a $P_{\lambda}$-point and $\mathfrak{b}^+\leq\lambda$.
\end{proof}
 Let us turn to our final application. We recall some definitions and theorems from \cite{Hrusak/Verner11}, and refer the reader to this paper for more details.

\begin{defn}[Definition $1.1$ of \cite{Hrusak/Verner11}]
    An ideal $I$ on $\omega$ is called \emph{locally $F_{\sigma}$} if for every $A\in I^+$, there is some $B\subseteq A$ in $I^+$ such that $I\upharpoonright B$ is $F_{\sigma}$.
\end{defn}
\begin{defn}[Definition $1.2$ of \cite{Hrusak/Verner11}]
A function $\mu:\mathcal{P}(\omega)\to \mathbb{R}^{\geq 0}\cup\{\infty\}$ is a lower semicontinuous submeasure (lscsm for short) on $\mathcal{P}(\omega)$ if the following holds:
\begin{enumerate}[label=(\roman*)]
    \item $\mu(\emptyset)=0$.
    \item If $A\subseteq B\subseteq \omega$, then $\mu(A)\leq \mu(B)$.
    \item If $A,B\subseteq \omega$, then $\mu(A\cup B)\leq \mu(A)+\mu(B)$.
    \item If $A\subseteq\omega$, then $\mu(A)=\lim_{n\to\infty}\mu(A\cap n)$.
\end{enumerate}
\end{defn}

Let us recall the following well-known theorem of Mazur \cite{Mazur}:

\begin{thm}
An ideal $I$ on $\omega$ is $F_{\sigma}$ if and only if there is some lscsm $\mu$ on $\mathcal{P}(\omega)$ with $I=\{A\subseteq \omega : \mu(A)<\infty\}$.
\end{thm}

Hru\v{s}\'ak and Verner characterized the definable ideals whose corresponding forcing adds a $p$-point in the following way:
\begin{thm}[Theorem $2.5$ of \cite{Hrusak/Verner11}]\label{thm:2.5ofHV}
    Suppose $I$ is an analytic ideal such that the forcing $\mathcal{P}(\omega)/I$ does not add reals. Then $\mathcal{P}(\omega)/I$ adds a $p$-point if and only if $I$ is locally $F_{\sigma}$.
\end{thm}

The following theorem will be used to answer Questions~\ref{Q: HrusakVerner}

\begin{thm}\label{thm.3.9}
     Let $I$ be an analytic ideal on $\omega$ such that $P(\omega)/I$ does not add reals. Then $P(\omega)/I$ adds an ultrafilter $U$ which is Tukey above $(\omega^\omega,\leq)$.
\end{thm}

\begin{proof}
Suppose first that $I$ is not locally $F_\sigma$. Then by \Cref{thm:2.5ofHV}, the generic ultrafilter $U$ added by $\mathcal{P}(\omega)/I$ is not a $p$-point; thus it is Tukey above $(\omega^\omega, \leq)$ this follows from Theorem~\ref{maintheorem}, but before that in \cite{TomCommute}. 
 
Assume now that $I$ is a locally $F_\sigma$ ideal. Let $U$ be the generic ultrafilter added by $\mathcal{P}(\omega)/I$. By density, there must be some $A^*\in U$ such that $I\restriction A^*$ is $F_\sigma$, and it suffices to show that $(U\restriction A^*, \supseteq)\geq_T (\omega^\omega, \leq)$. Let $\mu$ be an lscsm such that $I\restriction A^*=\{X\subseteq A^* : \mu(X)<\infty\}$. It follows that for every $X\in U\restriction A^*$ and every $n\in\omega$, there is some $m\in\omega$ such that $\mu(X\cap m)\geq n$. We now define a map $\varphi : U\restriction A^*\to \omega^\omega$. For $X\in U\restriction A^*$, we define a function $\varphi(X)\in\omega^\omega$ inductively as follows.
\begin{enumerate}[label=(\roman*)]
    \item $\varphi(X)(0)=\min(X)$.
    \item Suppose that $\varphi(X)(n)$ is already defined for some $n\in\omega$, and let $\varphi(X)(n+1)$ be the minimal natural number $m>\varphi(X)(n)$ such that $\mu(X\cap (\varphi(X)(n),m])>n+1$ (this must exist since $X\in (I\upharpoonright A^*)^+$). 
\end{enumerate}

First, we claim that $\varphi:(U\restriction A^*,\supseteq)\to(\omega^\omega, \leq)$ is monotone. Let $X\subseteq Y$ be arbitrary members of $U\restriction A^*$. To show that $\varphi(X)\geq \varphi(Y)$ we argue by induction. Clearly, we must have $\varphi(X)(0)\geq\varphi(Y)(0)$. Suppose now that we have $\varphi(X)(n)\geq\varphi(Y)(n)$ for some $n\in\omega$. Then for every $\varphi(X)(n)<m\in\omega$, we have $X\cap (\varphi(X)(n),m]\subseteq Y\cap (\varphi(Y)(n),m]$  and therefore, $$\mu(X\cap (\varphi(X)(n),m])\leq \mu(Y\cap (\varphi(Y)(n),m]).$$ It follows that $\varphi(X)(n+1)\geq \varphi(Y)(n+1)$.

To see that $\varphi:(U\restriction A^*, \supseteq)\to (\omega^\omega, \leq)$ is cofinal, let $f:\omega\to\omega$ be an arbitrary function. Let $X\subseteq A^*$ be any member of $(I\restriction A^*)^+$. We aim to show that there is some $X'\subseteq X$ with $X'\in (I\restriction A^*)^+$ such that $\varphi(X')\geq f$; once this is accomplished, $(U, \supseteq)\geq_T (\omega^\omega, \leq)$  follows from density. We first define a strictly increasing sequence of natural numbers $\langle m_n\rangle_{n\in\omega}$ by induction. Set $m_0=\min(X\setminus f(0))$. Now assume that $m_n$ is already defined for some $n\in\omega$. Set $m_n'=\max\{m_n,f(n+1)\}$, and define $m'_n< m_{n+1}\in\omega$ to be the minimal natural number $m$ such that $\mu(X\cap (m_n',m_{n+1}])>n+1$. We define \begin{equation}
X'=\{m_0\}\cup\bigcup_{n\in\omega}\left(X\cap(m_n',m_{n+1}]\right). 
\end{equation} It is straightforward to check that $\mu(X')=\infty$, and consequently $X'\in I^+$. Let us now show that $\varphi(X')\geq f$. By definition of $\langle m_n\rangle_{n\in\omega}$, it suffices to prove that $\varphi(X')(n)=m_n$ for all $n\in\omega$ (since $m_n$ is chosen so that $m_n\geq m_{n-1}'\geq f(n)$ for all $n\in\omega$). This clearly holds for $n=0$, so let us assume that $\varphi(X')(n)=m_n$ holds for some $n\in\omega$, and show that this is indeed the case for $n+1$ as well. By definition, $\varphi(X')(n+1)$ is the minimal natural number $m>\varphi(X')(n)=m_n$ such that $\mu(X'\cap (m_n,m])>n+1$ holds. But if $m\leq m_{n+1}$, then $X'\cap (m_n,m]=X\cap (m'_n,m]$ (as $X'$ does not intersect the interval $(m_n,m_n']$). Thus, $\varphi(X')(n+1)=m_{n+1}$ must be the case by definition of $m_{n+1}$, and we conclude the proof.
\end{proof}
    
By \Cref{thm:main characterization}, this answers the question of Hru\v{s}\'ak and Verner negatively:
\begin{cor}\label{cor: Answer question}
    There cannot be an analytic ideal $I$ such that $P(\omega)/I$ adds a Canjar ultrafilter.
\end{cor}
Another corollary gives a partial answers to \cite[Question 5.6]{Benhamou/Dobrinen24}:
\begin{cor}
    If $U$ is a generic ultrafilter added by $\mathcal{P}(\omega)/I$ for a locally $F_\sigma$ ideal $I$, then $U$ is Tukey-idempotent.
\end{cor}
\subsection{A remark about Canjar ultrafilters on large cardinals}\label{Section:Measurable}
The ideas presented in the previous section relate to a solution to another open problem:
\begin{question}[Brooke-Taylor {\cite[Question 3.6]{QuestionGeneralized}}]
Is there a $\kappa$-complete Canjar ultrafilter on a measurable cardinal $\kappa$?  Do Canjar
ultrafilters have a characterization using $p$-points?  
\end{question}
If the characterization from the previous section works at a measurable cardinal, then the question of finding a Canjar ultrafilter boils down to finding a $p$-point which is not Tukey idempotent. However, no such ultrafilter exists on a measurable by the result of the first two authors \cite[Theorem 6.7]{BENHAMOU_DOBRINEN_2024}. Instead of developing the theory to the measurable cardinals, we will argue directly that no such ultrafilter exists. The first step is to use the following observation:
\begin{thm}[Kanamori, Ketonen]\label{thm: extending the club filter}
Every $\kappa$-complete ultrafilter $U$ over a measurable cardinal $\kappa$ is Rudin-Keisler above an ultrafilter $W$ over $\kappa$ which extends the club filter.     
\end{thm}
For the proof, see for example \cite[Theorem 2.18]{MEASURES}. In particular, a $\kappa$-complete ultrafilter is going to be Tukey above the club filter $\text{Cub}_\kappa$. Also, it is well-known that for a regular cardinal $\kappa$, the club filter $\text{Cub}_\kappa$ is Tukey equivalent to $(\kappa^\kappa,\leq^*)$, where $\leq^*$ denoted the  domination order modulo the bounded ideal on $\kappa$. Finally, let us argue directly that the Mathias forcing with a $\kappa$-complete ultrafilter over $\kappa$ adds a dominating function $f:\kappa\to\kappa$.
\begin{thm}
    If $U$ is a $\kappa$-complete ultrafilter over $\kappa$, then $\mathbb{M}_U$ adds a function $f:\kappa\to\kappa$ such that for every $g:\kappa\to\kappa\in V$, $g\leq^* f$. In particular, $U$ is not Canjar.
\end{thm}
\begin{proof}
    Let $G$ be $\mathbb{M}_U$-generic. By Theorem~\ref{thm: extending the club filter}, there is a $\kappa$-complete ultrafilter $W$ such that $\text{Cub}_\kappa\subseteq W$ and $W\leq_{RK} U$. It is not hard to see now that $G$ induces a generic for $\mathbb{M}_W$ (e.g. see \cite{MEASURES}). Let $G^*$ be the induced generic and $X_{G^*}$ be the corresponding Mathias generic set. For each $\nu<\kappa$, let $\pi_{X_{G^*}}(\nu)$ be the $\nu^{\text{th}}$ element of $X_{G^*}$ in the increasing enumeration of $X_{G^*}$. Define the function $f(\alpha)=\pi_{X_{G^*}}(\alpha+1)$. We claim that for any $g:\kappa\to\kappa\in V$, $g\leq^* f$. Indeed, let $C_g$ be the club of closure points of $g$. Then $C_g\in W$, which then implies that there is $\xi<\kappa$ such that for for every $\xi<\alpha<\kappa$, $\pi_{X_{G^*}}(\alpha)\in C_g$. Take any $\xi\leq\alpha$. Since $\alpha\leq \pi_{X_{G^*}}(\alpha)<\pi_{X_{G^*}}(\alpha+1)$, and $\pi_{X_{G^*}}(\alpha+1)$ is a closure point of $g$, we obtain $g(\alpha)<\pi_{X_{G^*}}(\alpha+1)=f(\alpha)$.
\end{proof}
\begin{rem}
    Iterated Mathias forcing with Canjar ultrafilters were used in order to obtain models where $\mathfrak{b}<\mathfrak{a}$. On uncountable cardinals, it is open whether $\mathfrak{b}_\kappa<\mathfrak{a}_\kappa$ is consistent at all. Hence, the non-existence of Canjar ultrafilters on measurable cardinals rules out the possible of generalizing the original approach to measurable cardinals.
\end{rem}

\section{Ultrafilters arising from Topological Ramsey Spaces}\label{Section: TRS}

This section  investigates Tukey idempotency of  ultrafilters associated with topological Ramsey spaces. 
While such ultrafilters can be constructed under CH or MA, here we focus on ultrafilters forced by topological Ramsey spaces.
Investigations of ultrafilters forced by TRS's have their  nascence in the connection between Ramsey ultrafilters and the Ellentuck space, and   between stable ordered union ultrafilters and the space of infinite block sequences (see \cite{Mijares07}).
Investigations of ultrafilters forced by TRS's, as well as 
  their Tukey and Rudin-Keisler structures,  were carried out  in the series of papers  \cite{Dobrinen/Todorcevic14,Dobrinen/Todorcevic15,DobrinenJSL15,Dobrinen/Mijares/Trujillo14,DiPrisco/Mijares/Nieto17}.
  The concept of the $I$-p.i.p.\ in \cite{Benhamou/Dobrinen24} (recall Definition \ref{def: i-p.i.p}) was motivated by an understanding of these spaces, as Ramsey spaces have notions of diagonalization  so their associated ultrafilters behave like p-points, even though many of them are technically not p-points. 
Corollary 37 in \cite{Dobrinen/Todorcevic11} proved that  that all rapid p-points are Tukey-idempotent.
This implies that Ramsey ultrafilters are Tukey-idempotent, as well as the hierarchy of weakly Ramsey (and weaker partition relations) ultrafilters forced  Laflamme's forcings $\bP_\al$, $1\le \al<\om_1$, (see \cite{Laflamme89,Dobrinen/Todorcevic14,Dobrinen/Todorcevic15}), as well as ultrafilters forced by the TRS's in \cite{Dobrinen/Mijares/Trujillo14}, are Tukey idempotent.
However, that theorem does not apply to stable ordered union ultrafilters or to ultrafilters forced by $\mathcal{P}(\om^{\al})/\fin^{\otimes\al}$, for $2\le \al<\om_1$, which are also known to be forced by TRS's.
That those ultrafilters are Tukey-idempotent was shown in \cite{Benhamou/Dobrinen24},  as applications of 
Theorem \ref{Tukeyidempotency}.

In this section, we 
 prove that, under mild assumptions,   ultrafilters forced by topological Ramsey spaces
 are Tukey idempotent.  
The mild assumptions are satisfied in all known cases of ultrafilters forced by topological Ramsey spaces which have $\sigma$-closed separative quotients and add ultrafilters on the base set of first approximations.
We  mostly use standard notation for topological Ramsey spaces from Chapter 5 in   \cite{TodorcevicBK10}, a summary of which is  provided here.

\begin{defn}
Let $(\mathcal{R},\le ,r)$ be a triple, where $\le$ is a quasi-order on $\mathcal{R}$ and $r:\om\times\mathcal{R}\rightarrow \mathcal{AR}$ is  a sequence of restriction maps  $r_n= r(n,\cdot)$,  ($n\in\om$),  taking each $A\in\mathcal{R}$ to its $n$-th finite approximation $r_n(A)$.
\begin{enumerate}[label=(\roman*)]
    \item For $A\in\mathcal{R}$,  $[\emptyset,A]$  denotes the set $\{B\leq A : B\in\mathcal{R}\}$.
    \item For $n \in \omega$, $\mathcal{A}\mathcal{R}_n=\{r_n(X) : X \in \mathcal{R}\}$. $\mathcal{A}\mathcal{R}=\bigcup_{n\in\omega} \mathcal{AR}_n$  is the set of all finite approximations.
    \item For $A\in\mathcal{R}$, $\mathcal{AR}\restriction A=\bigcup_{n\in\omega}\mathcal{AR}_n\restriction A$, where $\mathcal{AR}_n\restriction A=\{r_n(B)\mid B\leq A\}$. 
    \item For $A\in\mathcal{R}$ and $a\in\mathcal{AR}$, we define $[a,A]=\{B\in\mathcal{R}\mid  B\leq A\mathrm{\ and \ }\exists n \, (r_n(B)=a)\}$.
    \item For simplicity we let $r_1[A]$ denote $\mathcal{AR}_1\restriction A$, for $A\in\mathcal{R}$.
    \item For $a\in \mathcal{AR}$ and $A\in\mathcal{R}$, we define $$\depth_A(a)=\begin{cases}
        \min(\{n\in\omega : a\leq_{\mathrm{fin}} r_n(A)\}), & \text{if there is such $n\in\omega$,} \\ \infty, & \text{otherwise,}
    \end{cases}$$ where the quasi-order $\leq_{\mathrm{fin}}$ on $\mathcal{AR}$ is the finitization of $\leq$.
\end{enumerate} 
\end{defn}

The sets $[a,A]$ defined in (iv)  are the basic open sets determining the {\em exponential} or {\em Ellentuck topology} on $\mathcal{R}$.
A subset $\mathcal{X}\sse\mathcal{R}$  is {\em Ramsey} if for each nonempty basic open set $[a,A]$, there is some $B\in[a,A]$ so that either $[a,B]\sse\mathcal{X}$ or else $[a,B]\cap\mathcal{X}=\emptyset$.
$\mathcal{X}$ is {\em Ramsey null} if the second case always happens.
A triple
$(\mathcal{R},\le,r)$
 is a {\em topological Ramsey space} iff every  subset of $\mathcal{R}$ with the property of Baire with respect to the  Ellentuck topology is  Ramsey,
 and every meager subset of $\mathcal{R}$ is Ramsey null.
 (See Definitions 5.2 and 5.3, \cite{TodorcevicBK10}.)

The Ellentuck space $([\om]^{\om},\sse, r)$ 
is the prototypical topological Ramsey space \cite{Ellentuck74}. 
Here, for $A\in[\om]^{\om}$, 
$r_n(A)$ is the set consisting of the least $n$ numbers in the set $A$.
It is well-known that  the forcing $([\om]^{\om},\sse^*)$ 
adds a Ramsey ultrafilter.
We point out that  $([\om]^{\om},\sse^*)$ is $\sigma$-closed  and is  forcing equivalent to  $\mathcal{P}(\om)/\fin$, which is the separative quotient of the Ellentuck space $([\om]^{\om},\sse)$ viewed as a forcing.
An analogous situation holds for most topological Ramsey spaces (see \cite{DiPrisco/Mijares/Nieto17} and \cite{DobrinenSEALS17}).
That is, given a  TRS $(\mathcal{R},\le, 
r)$, there is often  
 a coarsening $\le^*$ of $\le$ so that $(\mathcal{R},\le^*)$ is  $\sigma$-closed  and forcing equivalent to $(\mathcal{R},\le)$.
For any topological Ramsey space we may, without loss of generality,  
 assume that there is a greatest member, which we will denote by  $\mathbb{A}$;
 if no greatest member exists, just pick some $\mathbb{A}
 \in \mathcal{R}$ and work in $[\emptyset, \mathbb{A}]$ instead of $\mathcal{R}$.
 In all  TRS's referenced above, there is a  complete dense embedding of $(\mathcal{R},\le)$ 
into a  $\sigma$-closed Boolean algebra
 $\mathcal{P}(\mathcal{AR}_1)/I_{\mathcal{R}}$,
 where $I_{\mathcal{R}}$ is defined by 
$X\sse \mathcal{AR}_1$ is in $I_{\mathcal{R}}$  iff there is no $A\in\mathcal{R}$ such that $r_1[A]\sse X$.
The intended coarsening of $\le$ is via this ideal $I_{\mathcal{R}}$, or any ideal $I$ on $\mathcal{AR}_1$ so that  $\mathcal{P}(\mathcal{AR}_1)/I$ is the separative quotient of $(\mathcal{R},\le)$.

\begin{defn}
    Let $I\subseteq \mathcal{AR}_1$ be a proper ideal with the property  that for each $A\in\mathcal{R}$, $r_1[A]\not\in I$.
    We define the partial order $\leq_I$ on $\mathcal{R}$
    as follows: Given $A,B\in \mathcal{R}$, 
\begin{equation}
    A\leq_I B \text{ if and only if }   r_1[A]\setminus r_1[B]\in I.
\end{equation}
\end{defn}

Note that $\le_I$ is a coarsening of $\le$ on $\mathcal{R}$, as $A\le B$ implies $r_1[A]\sse r_1[B]$.
We will be concerned with  forcings $\mathbb{P}_I=(\mathcal{R},\le_I)$ and the generic filter that it adds
when $\mathbb{P}_I$ is forcing equivalent to $(\mathcal{R},\le)$.

\begin{defn}
    Let $G_I\subseteq \mathcal{R}$  be a $\mathbb{P}_I$-generic filter. 
    Define $U_I=\langle\{
    r_1[A] : A \in G_I\}\rangle$,
    the {\em first approximation filter}  induced by $G_I$ on 
$ \mathcal{P}(\mathcal{AR}_1)$.
\end{defn}

It is immediate that $U_I$ is a filter on the base set $\mathcal{AR}_1$.

\begin{lem}\label{Ipip}
    Suppose $\mathcal{R}$, $I$, and $U_I$ are as above. If $\mathbb{P}_I$ is separative and $\sigma$-closed, then $U_I$   has the $I$-p.i.p.
    Moreover, if $\mathbb{P}_I$ is forcing equivalent to $(\mathcal{R},\le)$, then $U_I$ is an ultrafilter.
\end{lem}

\begin{proof}
Suppose $A\in \mathbb{P}_I$ and there are $\mathbb{P}_I$-names $\dot{C}_n$ such that 
\begin{equation}
A\Vdash \text{``$\langle \dot{C}_n : n\in\omega\rangle$ is a $\leq_I$-decreasing sequence of members of $\dot{G}_I$"}.
\end{equation}
It suffices to show that there is some $B\leq_I A$ and a $\mathbb{P}_I$-name $\dot{C}$ such that $B\Vdash$ ``$\dot{C}\in \dot{G}_I$ and $\dot{C}\leq_I \dot{C}_n$,  for all $n\in\omega$".

Set  $A_0=A$ and define a decreasing sequence $\langle A_n : n \in \omega\rangle$ of members of $\mathcal{R}$, and a sequence $\langle B_n : n\in\omega\rangle$ of members of $\mathcal{R}$ such that for each $n\in\omega$, we have $A_{n+1}\Vdash B_n=\dot{C}_n$. After this, using the $\sigma$-closure of $\mathbb{P}_I$, we find some $A_{\infty}\in \mathbb{P}_I$ such that $A_{\infty}\leq_I A_n$ for all $n\in\omega$. But then $A_{\infty}\Vdash A_{\infty}\in \dot{G}_I$, and also $A_{\infty}\leq_I B_n$ for all $n\in \omega$ (by separativity of $\mathbb{P}_I$), which is what we wanted to show.

If  $\mathbb{P}_I$ is forcing equivalent to $(\mathcal{R},\le)$, then 
  genericity of $G_I$ and  a routine dense set argument, applying the fact that $\mathcal{R}$ is a TRS, shows that $U_I$ is an ultrafilter.
\end{proof}



Our goal is to prove that, under mild hypotheses, $U_I$ is a Tukey idempotent ultrafilter. Let us now introduce one of the ideals we will utilize. In \cite{Mijares07}, Mijares defined the following notion for $a\in \mathcal{AR}$:
\begin{equation}
    \depth^0_{\mathbb{A}}(a)=\max\{n\leq \depth_A(a)\mid \forall b\leq_{\fin} r_n(A), \exists B\in [b,A] ([a,B]\neq\emptyset)\},
\end{equation}
where $\leq_{\fin}$ is a quasi-order  on $\mathcal{AR}$, defined by $b\le_{\fin} a$ iff
$b=r_n(B)$ and $a=r_m(A)$
for some $B\le A$ in $\mathcal{R}$ and $n\le m=\depth_A(b)$.
We define the ideal $I^0_{\mathcal{R}}=\{X\subseteq \mathcal{AR}_1\mid (\exists N\in\omega) \ (\forall a\in X)\ \depth^0_{\mathbb{A}}(a)<N\}$. 
It is straightforward to verify that $I^0_{\mathcal{R}}\equiv_T\omega$.
Mijares showed in \cite{Mijares07} that whenever $(\mathcal{R},\le_{I^0_{\mathcal{R}}})$ is $\sigma$-closed, then this forcing adds an ultrafilter on $\mathcal{AR}_1$.

\begin{defn}\label{defn.star}
The following is our Assumption $(*)$:
\begin{enumerate}[label=(\roman*)]
    \item $\mathbb{P}_I$ and $(\mathcal{R}, \leq)$ have isomorphic separative quotients, and $\leq_I$ is a $\sigma$-closed partial order.\label{maintheoremassmpt1}
    \item 
    $I^0_{\mathcal{R}}\subseteq I$,  $I^\omega\leq_T\omega^\omega$, and there is no $A\in\mathcal{R}$ with $r_1[A]\in I$.\label{maintheoremassmpt2}
    \item Either $I=I^0_{\mathcal{R}}$ 
    and
 the Rudin-Keisler projection via depth map on $U_{I^0_{\mathcal{R}}}$ 
 yields a Ramsey ultrafilter;
    or for all $A\in\mathcal{R}$, there is $B\le A$ such that $I^0_{\mathcal{R}|B}$ is not $\sigma$-closed.\label{maintheoremassmpt3}
    \end{enumerate}
\end{defn}

(i) in Assumption $(*)$
ensures that the forcing $\mathbb{P}_I$ does not add new subsets of $\mathcal{AR}_1$. Consequently, the first approximation filter $U_I$ is an ultrafilter on the base set $\mathcal{AR}_1$. The two alternatives given in \ref{maintheoremassmpt3} should be thought of as a dichotomy,
where the higher and infinite dimensional Ellentuck spaces satisfy the latter part of \ref{maintheoremassmpt3}, and  all other known  TRS's which have $\sigma$-closed separative quotients satisfy the former.
In particular, all TRS's with the ISS or IEP properties, defined in \cite{Dobrinen/NavarroFlores}, satisfy the first clause of (iii) (see Theorem 2.1 \cite{NavarroFloresThesis}, which guarantees the Ramsey ultrafilter RK below $U_I$ under these hypotheses).


We now state our result.

\begin{thm}\label{maintheorem}
Let $\mathcal{R}$ be a topological Ramsey space and  let $I\subseteq \mathcal{AR}_1$ be a proper ideal  so that $(*)$ holds.
Then the generic first approximation filter $U_I$ for $\mathcal{R}$ is a Tukey idempotent ultrafilter.
\end{thm}

This theorem recovers all previously known  examples of Tukey-idempotent ultrafilters forced by topological Ramsey spaces, and additionally shows the following:

\begin{cor}
All Milliken-Taylor ultrafilters forced by $\FIN^{[\infty]}_n$, $1\le n<\om$,
  are Tukey idempotent.
\end{cor}

\begin{rem}
Corollary 1.9 of \cite{Benhamou/Dobrinen24} showed that if $U$ and $V$ are each Tukey-idempotent, then $U\cdot V\equiv_T V\cdot U$.
This together with Theorem  \ref{maintheorem}
shows commutativity of  Fubini products of any ultrafilters forced by reasonable (i.e., satisfying $(*)$) Ramsey spaces, up to Tukey equivalence.   
\end{rem}

We are now ready to prove \Cref{maintheorem}.

\begin{proof}[Proof of \Cref{maintheorem}]
We will utilize \Cref{Tukeyidempotency}. The proof is done in two cases, according to which part of \ref{maintheoremassmpt3} our space satisfies.

\textit{Case 1.} First suppose that $I=I^0_\mathcal{R}$ 
is $\sigma$-closed and 
 the Rudin-Keilser projection via depth map on $U_{I^0_{\mathcal{R}}}$ 
 yields a Ramsey ultrafilter.
 By \Cref{Ipip}, we know that $U_I$ has the $I$-p.i.p. Moreover, since $U_I$ is Rudin-Keisler above a Ramsey ultrafilter, 
 we also know that $U_I\geq_T\omega^\omega$. Since $I^0_{\mathcal{R}}\equiv_T\omega$, we can apply \Cref{Tukeyidempotency} to conclude this case.

\textit{Case 2.} Now suppose that the first two assumptions of $(*)$ hold, and moreover that for all $A\in\mathcal{R}$, there is $B\le A$ such that $I^0_{\mathcal{R}|B}$ is not $\sigma$-closed. Note that by \Cref{Ipip}, we already know that $U_I$ has the $I$-p.i.p. So it suffices to show that $U_I\geq_T I^{\omega}$ holds. A standard density argument and the first part of assumption \ref{maintheoremassmpt2} of $(*)$ yields a  sequence 
$A_0 \ge_{I^0_{\mathcal{R}}}   A_1\ge_{I^0_{\mathcal{R}}}  A_2\ge_{I^0_{\mathcal{R}}}\dots$
 of members of $G_I$, the generic filter forced by $(\mathcal{R},\le_I)$,  which does not have a pseudo-intersection with respect to the ideal $I^0_{\mathcal{R}}$. For $n\in\omega$, define
\begin{equation}
X_n=\{a\in\mathcal{AR}_1 : a\notin r_1[A_n]\}.
\end{equation}
Note that if for some $n\in\omega$ and $B\in\mathcal{R}$ we have $|r_1[B]\cap X_n|<\aleph_0$, then it must be the case that $B\leq_{I^0_{\mathcal{R}}} A_n$. Moreover, there is no $C\in G_I$ such that $r_1[C]\subseteq \bigcap_{n\in\omega} X_n^c$; for otherwise $C\leq_{I^0_{\mathcal{R}}} A_n$ would hold for all $n\in\omega$. Similarly, there is no $C\in G_I$ with $|r_1[C]\cap X_n|<\aleph_0$ for all $n\in\omega$. It follows that $U_I$ is not a $p$-point (in the classical sense). Thus, by
Theorem $4.2$ of \cite{TomCommute}, 
we have $U_I\geq_T \omega^\omega\geq_T I^\omega$, which finishes the proof.
\end{proof}

\section{Further directions \& open problems}\label{Section: Questions} We conclude this paper with three open problems.
\begin{question}
    Is there a notion of $I$-Canjar/$I$-strong-p.i.p.\ that characterizes Tukey-idempotency inside the class of $I$-p.i.p.\ ultrafilters?
\end{question}
\begin{question}
    Is there an ideal $I$ such that $P(\omega)/I$ does not add reals and the generic ultrafilter $U$ is not Tukey-idempotent?
\end{question}
\begin{question}
    Is there an ultrafilter forced by  a topological Ramsey space  not adding reals which is not Tukey-idempotent?
    
\end{question}

\bibliographystyle{plain}
\bibliography{references}

\end{document}